\documentclass{amsart}                

\usepackage{amssymb,amsmath}

\usepackage{hyperref}

\usepackage{color}

\theoremstyle{plain}
\newtheorem{theorem}{Theorem}[section]
\newtheorem*{theorem*}{Theorem}
\newtheorem{lemma}[theorem]{Lemma}

\newtheorem{proposition}[theorem]{Proposition}
\theoremstyle{remark}
\newtheorem{remark}[theorem]{Remark}
\numberwithin{equation}{section}
\theoremstyle{definition}
\newtheorem{definition}[theorem]{Definition}
\numberwithin{equation}{section}
\numberwithin{equation}{section}

\newcount\quantno
\everydisplay{\quantno=0}\everycr{\quantno=0}
\newcommand\quant{\advance\quantno by1
                      \ifnum\quantno=1\qquad\else\quad\fi\forall }
\newcommand\itemno[1]{(\romannumeral #1)}

\renewcommand\Re{\operatorname{\mathrm{Re}}}
\renewcommand\Im{\operatorname{\mathrm{Im}}}
\newcommand\set[1]{{\left\{#1\right\}}}

\renewcommand\mod[1]{\left\vert{#1}\right\vert}
\newcommand\bigmod[1]{\bigl\vert{#1}\bigr|}

\newcommand\norm[2]{{\Vert{#1}\Vert_{#2}}}

\newcommand\bignorm[2]{\left.{\bigl\Vert{#1}\bigr\Vert_{#2}}\right.}
\newcommand\opnorm[2]{|\!|\!| {#1} |\!|\!|_{#2}}
\newcommand\bigopnorm[2]{\bigl|\!\bigl|\!\bigl| {#1} 
\bigr|\!\bigr|\!\bigr|_{#2}}
\newcommand\Bigopnorm[2]{\Bigl|\!\Bigl|\!\Bigl| {#1} 
\Bigr|\!\Bigr|\!\Bigr|_{#2}}
\newcommand\prodo[2]{\left\langle#1,#2\right\rangle}

\newcommand\smallfrac[2]{\mbox{\small$\displaystyle\frac{#1}{#2}$}}
\newcommand\wrt{\,\text{\rm d}}

\newcommand\opL{\operatorname{\mathcal{L}}}

\newcommand\bD{\mathbf{D}}

\newcommand\bH{\mathbf{H}}

\newcommand\bS{\mathbf{S}} 

\newcommand\BC{\mathbb{C}}

\newcommand\BN{\mathbb{N}}
\newcommand\BR{\mathbb{R}} 

\newcommand\BZ{\mathbb{Z}}

\newcommand\cA{\mathcal{A}}  
\newcommand\cB{\mathcal{B}}  

\newcommand\cD{\mathcal{D}} 

\newcommand\cF{\mathcal{F}}  

\newcommand\cH{\mathcal{H}}  

\newcommand\cJ{\mathcal{J}}

\newcommand\cL{\mathcal{L}}  
  
\newcommand\fM{\mathfrak{M}} 

\newcommand\cO{\mathcal{O}}  
\newcommand\cP{\mathcal{P}}  

\newcommand\cS{\mathcal{S}} 

\newcommand\cT{\mathcal{T}}

\newcommand\al{\alpha}
\newcommand\be{\beta}
\newcommand\ga{\gamma}    
\newcommand\Ga{\Gamma}
\newcommand\de{\delta}
 
\newcommand\vep{\varepsilon}

\newcommand\la{\lambda}   
\newcommand\La{\Lambda}
\newcommand\om{\omega}    

\newcommand\vp{\varphi}
\newcommand\OV{\overline}
\newcommand\lu[1]{L^1(#1)}
\newcommand\lp[1]{L^p(#1)}

\newcommand\ld[1]{L^2(#1)}

\newcommand\ly[1]{L^\infty(#1)}
\newcommand\lorentz[3]{L^{#1,#2}(#3)}
\newcommand\hu[1]{H^1(#1)}

\newcommand\wh{\widehat}
\newcommand\wt{\widetilde}

\newcommand\ds{\displaystyle}

\newcommand\rmi{\hbox{\rm (i)}}
\newcommand\rmii{\hbox{\rm (ii)}}
\newcommand\rmiii{\hbox{\rm (iii)}}

\newcommand\ir{\int_{-\infty}^{\infty}}
\newcommand\ioty{\int_0^{\infty}}
\newcommand\dtt[1]{\,\frac{\mathrm {d} #1}{ #1}}

\newcommand\One{{\mathbf{1}}}

\newcommand\e{\mathrm{e}}

\newcommand\Horm{\mathrm{Horm}}
\newcommand\GK{{G/K}}

\newcommand\dest{\text{\rm d}}
\newcommand\bL{\mathbf{L}}

\newcommand\Ric{\mathop{\rm Ric}}
\newcommand\LL[1]{L^2(\Lambda^{#1}_\BC M)}

\begin{document}

\title[Weak type $1$ estimates]
{Estimates for functions 
of the Laplacian\\
on manifolds with bounded geometry}



\thanks{Work partially supported by the
Italian Progetto Cofinanziato ``Analisi Armonica'', 2008--2009.}

\author[G. Mauceri, S. Meda \and M. Vallarino]
{Giancarlo Mauceri, Stefano Meda and Maria Vallarino}

\address{Giancarlo Mauceri \\
Dipartimento di Matematica\\
Universit\`a di Genova\\
via Dodecaneso 35 \\
16146 Genova, Italy}
\email{mauceri@dima.unige.it}

\address{Stefano Meda: 
Dipartimento di Matematica e Applicazioni
\\ Universit\`a di Milano-Bicocca\\
via R.~Cozzi 53\\ I-20125 Milano\\ Italy}
\email{stefano.meda@unimib.it}

\address{Maria Vallarino:
Laboratoire de Math\'ema\-ti\-ques et Applications, 
Physique Math\'ema\-ti\-ques d'Orl\'eans\\
 Universit\'e d'Orl\'eans, UFR Sciences\\
B\^atiment de Math\'ematique-Route de Char\-tres\\
B.P. 6759\\ 45067 Orl\'eans cedex 2\\ France}
\email{maria.vallarino@unimib.it}

\begin{abstract}
In this paper we consider a complete connected noncompact 
Riemannian manifold $M$ with Ricci curvature bounded from below
and positive injectivity radius.  Denote by $\cL$
the Laplace--Beltrami operator on $M$.  We assume that the kernel
associated to the heat semigroup generated by $\cL$
satisfies a mild decay condition at infinity.  
We prove that if~$m$ is a bounded holomorphic function 
in a suitable strip of the complex plane, and satisfies
Mihlin--H\"ormander type conditions of appropriate order at infinity, 
then the operator $m(\cL)$ extends to an operator of weak type~$1$.
\par
This partially extends a celebrated result of 
J.~Cheeger, M.~Gromov and M.~Taylor,
who proved similar results under much stronger curvature
assumptions on $M$, but without any assumption on the decay
of the heat kernel.   
\end{abstract}

\maketitle

\setcounter{section}{0}
\section{Introduction} \label{s:Introduction}

The purpose of this paper is to extend a celebrated multiplier result 
of J.~Cheeger, M.~Gromov and M.~Taylor \cite[Thm~10.2]{CGT},
\cite[Thm~1.6]{T}, by substantially relaxing its geometric assumptions.

Suppose that $M$ is a complete connected noncompact Riemannian manifold.
Denote by~$-\cL$ the Laplace--Beltrami operator on $M$: 
$\cL$ is a symmetric operator on $C_c^\infty(M)$ (the space of
compactly supported smooth complex-valued functions on $M$).  Its closure
is a self adjoint operator on $\ld{M}$ which, with a slight abuse of notation,
we denote still by $\cL$. 
We denote by $b$ the bottom of the spectrum of $\cL$,
and by~$\{\cP_{\la}\}$ the 
spectral resolution of the identity for which   
$$
\cL f = \int_b^\infty \la \wrt \cP_{\la}f
$$
for every $f$ in the domain of $\cL$.  
For notational convenience, we denote by $\cD$ the
operator~$\sqrt{\cL-b}$.
\par
We say that $M$ has \emph{$C^\infty$ bounded geometry} if
the injectivity radius of $M$ is positive and the Riemann 
curvature tensor is bounded in the~$C^{\infty}$ topology.
We say that $M$ has \emph{bounded geometry} if
the injectivity radius of $M$ is positive and the Ricci
curvature is bounded from below.
 If $M$ has bounded geometry,
then there are nonnegative constants $\al$, $\be$ and $C$ such that
\begin{equation} \label{f: volume growth}
\mu\bigl(B(x,r)\bigr)
\leq C \, r^{\al} \, \e^{2\be \, r}
\quant r \in [1,\infty)\quad\forall x\in M,
\end{equation}
where $\mu\bigl(B(x,r)\bigr)$ denotes the Riemannian volume of the 
geodesic ball with centre $x$ and radius~$r$.  The same 
is \emph{a fortiori} true if $M$ has $C^\infty$ bounded geometry.   

For each ~$W$ in $\BR^+$, denote by $\bS_{W}$ the strip
$
\bigl\{z \in \BC: \Im{z} \in (-W,W) \bigr\}
$.
Taylor \cite[Thm~1.6]{T}, following up earlier work
of Cheeger, Gromov and Taylor \cite[Thm~10.2]{CGT}, proved that 
if $M$ has $C^\infty$ bounded geometry and
$m$ is a bounded even holomorphic function
in $\bS_{\be}$ satisfying estimates of the form
\begin{equation} \label{f: pseudodiff estimates 1} 
  \mod{D^{j} m(\zeta)}
\leq C\, \bigl(1+\mod{\zeta}\bigr)^{-j} 
\quant \zeta \in \bS_{\beta} \quant j \in \{0,1,\ldots, J\},
\end{equation}
where $J$ is a sufficiently large integer depending on the
dimension $n$ of $M$, then the operator~$m (\cD)$ is bounded
on $\lp{M}$ for $p$ in $(1,\infty)$, and of weak type~$1$.

In fact, the proof of this result requires
control only of a finite number of covariant
derivatives of the Riemann tensor, but this number is of the same
order of magnitude of the dimension of $M$.
Notice that \cite[Thm~1.6]{T}  
extends a previous result of R.J.~Stanton and P.A.~Tomas
\cite{StTo} in the case where $M$ is a symmetric space
of the noncompact type $\GK$ and real rank one
(and $\cL$ is the Laplace--Beltrami operator associated to
the $G$-invariant metric on~$M$ induced by the Killing form of $G$).
See also the pioneering
work of J.L.~Clerc and E.M.~Stein \cite{CS} on spherical multipliers
on noncompact symmetric spaces associated to complex semisimple
Lie groups, and recent related works \cite{A,I1,I2,MV} on general
noncompact symmetric spaces, which have been stimulated by 
\cite[Thm~1.6]{T} and \cite{CS}. 

As M.~Berger says in his book \cite[p.~291]{Be} ``Up to the
end of the 1980s, Ricci curvature was believed to be only
useful to control volumes,... ''.  Since then,
various geometric and analytic results on Riemannian manifolds have
been established under the hypothesis that the manifold is
of bounded geometry.  To mention a few, we recall the
relationship between isoperimetric
inequalities and the behaviour for large time
of the heat kernel \cite{Cou,ChF,V} 
and local Harnack type estimates for positive solutions
of the heat equation \cite{SC}.

In view of these considerations, 
it is natural to speculate whether \cite[Thm~1.6]{T} 
may be extended to Riemannian manifolds of bounded geometry.   
In this paper we assume that~$M$ has bounded geometry in
the sense specified above, but, for technical reasons, we
need also 
to assume that there exist constants $\rho> 1/2$ and $C$ such that 
\begin{equation} \label{f: assumption Ht}
\opnorm{\cH^t}{1;\infty}
\leq C \, \e^{-bt} \, t^{-\rho} \quant t \in [1,\infty),
\end{equation}
where $\{\cH^t\}$ denotes the heat semigroup generated by $\cL$,
and $\opnorm{\cH^t}{1;\infty}$ the 
norm of $\cH^t$ \emph{qua} operator from $\lu{M}$ to $\ly{M}$.
Note that on every manifold $M$ with bounded geometry estimate (\ref{f: assumption Ht}) holds, but with~$\rho =0$
(see, for instance, \cite[Section~7.5]{Gr}).  
Moreover it holds with $\rho>1$ on nonamenable unimodular Lie groups with a left invariant Riemannian metric \cite{Lo} and on noncompact Riemannian symmetric spaces \cite{CGM}.
\par
Our main result, Theorem~\ref{t: multiplier 2}, states that 
if $M$ is a Riemannian manifold 
of bounded geometry satisfying (\ref{f: assumption Ht}) with $\rho>1/2$, 
then the conclusion of Taylor's result holds with
$J> \max([\!\![n/2+1]\!\!] + 2, [\!\![n/2+1]\!\!] + 2+\al/2-\rho)$.  

To prove Theorem~\ref{t: multiplier 2} 
we decompose, as in \cite[Thm~1.6]{T}, the Schwartz kernel $k_{m(\cD)}$ of $m(\cD)$
as the sum of a kernel with support near the diagonal
in $M\times M$, and of a kernel supported off the diagonal.
As in \cite{T,CGT}, we show that the part near the diagonal
satisfies a H\"ormander type integral condition,
and that the part off the diagonal gives rise to a bounded
operator on $\lu{M}$.  However, the technical details are rather
different.  In particular, since we do not assume any control on
the derivatives of the Riemann tensor, we cannot use
either the eikonal equation or the Hadamard parametrix construction
to obtain the required estimates of $k_{m(\cD)}$ near the diagonal.
Our approach to these estimates is
based on ultracontractive bounds for the heat semigroup and 
for the restriction of the semigroup
generated by the de~Rham operator to $1$-forms on $M$ and uses an adaptation of L. H\"ormander's method \cite{Ho}.
We believe that our approach, though technically elaborate, helps to understand and clarify the relationships
between the heat semigroup and singular integral operators on $M$.  

We recall that the idea to use ultracontractive estimates
for the heat semigroup in the proof of multiplier results 
for its generator is not new (see, for instance, \cite{CoSi}), but,
to the best of our knowledge, it is indeed new in our setting.

Different endpoint estimates for various classes of multiplier operators
on manifolds with bounded geometry will be considered
in a forthcoming paper \cite{MMV}.  Those estimates will involve 
the Hardy space $\hu{M}$ introduced in \cite{CMM} 
and some related spaces, which will be defined in \cite{MMV}.

We will use the ``variable constant convention'', and denote by $C,$
possibly with sub- or superscripts, a constant that may vary from place to 
place and may depend on any factor quantified (implicitly or explicitly) 
before its occurrence, but not on factors quantified afterwards. 

\section{Notation, background material and preliminary results}
\label{s: New Hardy}

Suppose that $M$ is a connected $n$-dimensional Riemannian manifold
of infinite volume with Riemannian measure $\mu$.
We assume that $M$ has bounded geometry, i.e., that the injectivity
radius of
$M$ is positive and that 
\begin{equation}\label{f: Ricci}
\Ric(M)\ge-\kappa^2
\end{equation}
for some  $\kappa\ge0$. It is well known that manifolds with bounded geometry
satisfy the \emph{uniform ball size condition}, i.e. for every $r \in \BR^+$ 
\begin{equation}\label{ubc}
\inf\, \bigl\{ \mu\bigl(B(x, r)\bigr): x \in M \bigr\} > 0, \qquad
\sup\, \bigl\{ \mu\bigl(B(x, r)\bigr): x \in M \bigr\} < \infty
\end{equation}
(see, for instance, \cite{CMP}, where complete references are given). Moreover, by standard comparison theorems \cite[Theor. 3.10]{Ch}, the measure $\mu$ is \emph{locally doubling}, i.e. for every $R>0$ there exists a constant $C_R$ such that for every ball $B(x,r)$ such that $r<R$
$$
\mu\big(B(x,2r)\big)\le\,C_R\,\mu\big(B(x,r)\big).
$$
\par
If $\cT$ is a bounded linear operator from $L^p(M)$ to $L^q(M)$, we shall denote by~$\opnorm{\cT}{p;q}$ 
the operator norm of $\cT$ from $L^p(M)$ to $L^q(M)$.  In the
case where $p=q$, we shall simply write
$\opnorm{\cT}{p}$ instead of $\opnorm{\cT}{p;p}$. \par
Denote by $-\cL$ the Laplace--Beltrami operator
on $M$, 
by $b$ the bottom of the $\ld{M}$ spectrum of $\cL$,
and set $\be =
\limsup_{r\to\infty} \bigl[\log\mu\bigl(B(o,r)\bigr)\bigr]/(2r)$.
By a result of R.~Brooks \cite{Br}  $b\leq \be^2$.  
Denote by $\{\cH^t\}$ the semigroup generated by $-\cL$ on $\ld{M}$.
It is well known that for every $p$ in $[1,\infty)$,
the operator $\cH^t$ extends  
to a contraction on $\lp{M}$.
Furthermore, $\{\cH^t\}$ is ultracontractive, i.e., $\cH^t$
 maps $\lu{M}$ into $\ly{M}$ for every $t$ in $\BR^+$.
Recall that $\cH^t$ satisfies the following 
estimate \cite[Section~7.5]{Gr}  
\begin{equation} \label{f: special}
\bigopnorm{\cH^t}{1;2}
\leq C\, \e^{-bt} \, t^{-n/4} \, (1+t)^{n/4-\rho/2}  \quant t \in \BR^+
	\end{equation}
for some $\rho$ in $[0,\infty)$.  
Then, by standard subordination techniques of semigroups
\begin{equation} \label{f: ultrac bounds}
\bigopnorm{\e^{-t\cD}}{1;2} 
\leq 
C\, t^{-n/2}  \,(1+t)^{n/2-\rho}  \quant t \in \BR^+.
\end{equation}

The proof of our main result, Theorem~\ref{t: multiplier 2} below,
is rather technical and requires some
background material and a few preliminary results,
which are the content of the following three subsections.  
Specifically, 
Subsection~\ref{subsec: Background} gives information
about the heat semigroup generated by~$\cL$, and its natural
extension to forms, i.e., the semigroup generated by the 
de~Rham operator $\bL$,
Subsection~\ref{subsec: Kernel} and Subsection~\ref{subsec: Notation} contain
estimates for the kernels of certain functions
of the operator $\cD$
and some technical lemmata respectively.  These results
will be directly used in the proof of Theorem~\ref{t: multiplier 2}.

\subsection{The de~Rham operator} 
\label{subsec: Background}

Denote by $L^2(\La^k_\BC M)$ the space of square integrable 
$k$-forms on $M$ with complex coefficients, and by 
$C_c^\infty(\La^k_\BC M)$ the subspace of smooth forms with compact support. 
As usual we identify $0$-forms with functions on $M$,
and $L^2(\La^0_\BC M)$ with the space $L^2(M)$.   We shall denote by 
$\langle\cdot,\cdot\rangle_k$ the Hermitian inner product on 
$L^2(\La^k_\BC M)$, i.e.
$$
\langle\om,\eta\rangle_k
= \int_{M} \bigl( \om(y),\eta(y)\bigr)_y \wrt \mu(y),
$$
where $(\cdot,\cdot)_y$ is the 
Hermitian inner product  induced by the metric of 
$M$ on the complexification of the space of alternating tensors of order 
$k$ at the point $y$. 

We denote by $\dest$ the operator of exterior differentiation, 
considered as a closed densely defined operator from 
$L^2(\La^k_\BC M)$ to $L^2(\La^{k+1}_\BC M)$ and by 
$\de$ its adjoint operator, i.e. the closed densely defined operator mapping 
$L^2(\La^{k+1}_\BC M)$ to $L^2(\La^k_\BC M)$ such that
$$
\langle\dest \om,\eta\rangle_{k+1}
=\langle \om,\de\,\eta\rangle_k \quant \om\in C^\infty_c(\La^k_\BC M)
\quant \eta \in C^\infty_c(\La^{k+1}_\BC M).
$$
As a consequence, for each nonnegative integer $k$
the de~Rham operator $\de\dest+\dest\de$ maps 
smooth $k$-forms into smooth $k$-forms. 

We denote by $\bL$ the operator on 
$\bigoplus_{k=0}^\infty C^\infty_c(\La^{k}_\BC M)$, defined by
$$
\bL \om = (\de\dest+\dest\de) \om 
\quant \om \in C^\infty_c(\La^{k}_\BC M).
$$
With a slight abuse of notation, the closure of 
$\bL$ in $\bigoplus_{k=0}^\infty \ld{\La^{k}_\BC M}$
will be denoted still by~$\bL$.

It is well known that for each
nonnegative integer $k$,
the restriction of $\bL$ to $\ld{\La^{k}_\BC M}$ is a self adjoint operator
\cite{S}.
In particular,  the restriction of $\bL$ to $L^2(\La^0_\BC M)$ 
coincides with $\cL$.  
Furthermore, it is known that the restriction of $\bL$
to $1$-forms is nonnegative.  Therefore, the 
restriction of $-\bL$ to $\ld{\La^{1}_\BC M}$ generates a strongly continuous 
one parameter semigroup on $L^2(\La^1_\BC M)$ that 
we denote by $\{\bH^t\}$. 

The next lemma summarises some of the properties of the operator 
$\bL$ on $1$-forms that we shall need in the proof
of Proposition \ref{p:kernest}. 

\begin{lemma} \label{l: propLex} Suppose that $\kappa$ is as in (\ref{f: Ricci}).
Then
\begin{enumerate}
\item[\itemno1] for every bounded Borel function $F$ on $[0,\infty)$
$$
F(\cL)\de \om
= \de F(\bL) \om
\quant \om \in C^\infty_c(\La^{1}_\BC M);
$$
\item[\itemno2]  for every $\om$
in $\LL{1}$
$$
\mod{\bH^t \om(x)}_x
\le \e^{\kappa^2 t}\, \cH^t \mod{\om}(x) 
\quant t\in \BR^+ \quant x\in M;
$$
\item[\itemno3]
 for every $\om$
in $\LL{1}$
$$
\norm{\bH^t\om}{2}
\le \e^{(\kappa^2-b)t}\norm{\om}{2}
\quant t \in \BR^+.
$$
Hence the bottom of the spectrum of $\bL$ on 
$\LL{1}$ is greater than or equal to $b-\kappa^2$. 
\end{enumerate}
\end{lemma}

\begin{proof} 
$\rmi$
The identity $F(\cL)\de =\de F(\bL)$ 
is a straightforward consequence of the identity $\cL\de\, \om
=\de\bL\,\om$ and of the fact that the operators $\cL$ and 
$\bL$ are self-adjoint.

For the proof of $\rmii$ see \cite[Prop. 1.7]{Ba};
now \rmiii\ follows directly from \rmii.
\end{proof}

Suppose that $\om$ is a smooth $1$-form with compact support. From Lemma 2.1, by using the contractivity of $\cH^t$ on $L^p(M)$ for every $p\in[1,\infty]$ and interpolation,
one may deduce that if $p$ is in $[1,\infty]$, then
$$
\bignorm{\bH^t \om}{p}
\leq \e^{\kappa^2\mod{1-2/p} t}\, \norm{\om}{p} \quant t\in \BR^+.
$$
Thus, $\{\bH^t\}$ extends to a semigroup on $L^p({\La^1_\BC M})$ for all $p$ in $[1,\infty)$,
that we denote still by~$\{\bH^t\}$.
  Fr{}om 
Lemma~\ref{l: propLex} and the ultracontractivity
estimates (\ref{f: special}) for $\cH^t$ we also deduce that 
$$
\begin{aligned}
\bignorm{\bH^t \om}{2}
&  =   \Bigl[\int_M \mod{\bH^t \om(x)}_x^2 \wrt \mu(x) \Bigr]^{1/2} \\
&  \le \e^{\kappa^2 t}\, \Bigl[\int_M 
     \bigl(\cH^t \mod{\om}(x)\bigr)^2 \wrt \mu(x) \Bigr]^{1/2} \\
&  \le C \, \e^{(\kappa^2-b) t}\,  t^{-n/4}\, (1+t)^{n/4-\rho/2} 
	 \, \norm{\om}{1} \quant t\in \BR^+.
\end{aligned}
$$
Thus there exists a constant $C$ such that 
\begin{equation}\label{ucLex}
\bigopnorm{\bH^t}{1;2}
\leq C\, \e^{(\kappa^2-b) t} \, t^{-n/4} \, (1+t)^{n/4-\rho/2}  
\quant t \in \BR^+.
\end{equation}

\subsection{Estimates for certain kernels} \label{subsec: Kernel}

For notational convenience, we denote by $\cD_1$
the operator $\sqrt{\cL-b+\kappa^2}$, and by $\bD_1$
the operator $\sqrt{\bL-b+\kappa^2}$ (the operator $\bL-b+\kappa^2$ is nonnegative by Lemma \ref{l: propLex}\rmiii).

If $\cT$ is an operator bounded on $\ld{M}$, then we
denote by $k_{\cT}$ its Schwartz kernel (with respect to the
Riemannian density $\mu$).
In this subsection, 
we prove estimates for $k_{F(t \cD)}$, $k_{F(t \cD_1)}$  
and of $\dest_2 k_{F(t \cD_1)}$, when the function~$F$ 
decays sufficiently fast at infinity; here $\dest_2$ denotes the differential with respect to the second variable. \par We observe that the only reason  to
introduce the de Rham operator $\bL$ and the auxiliary operator $\bD_1$ is that 
to estimate the kernel of 
$\dest_2 k_{F(t \cD_1)}$ we exploit the identity in Lemma \ref{p:kernest} \rmi.
\begin{proposition} \label{p:kernest}
Assume that $\rho$ is as in (\ref{f: assumption Ht}), 
that $\ga$ is in $(n/2+1,\infty)$, and that~$F$ 
is a bounded  function on $[0,\infty)$ such that
$$
\sup_{\la\in\BR^+}\, \mod{\big(1+{\la}\big)^{\ga} \, F(\la)}
< \infty.
$$
Then for every $t$ in $\BR^+$ the operators $F({t}\cD)$,
$F({t}\cD_1)$ and $\dest F({t}\cD_1)^*$ are bounded from 
$L^1(M)$ to $L^2(M)$.   Furthermore, there exists
a constant $C$ such that for all $t$ in $\BR^+$
\begin{enumerate}
\item[\itemno1]
$
\sup_{y\in M} \, \, \bignorm{k_{F(t\cD)}(\cdot,y)}{2} 
\leq C\, t^{-n/2}\, (1+t)^{n/2-\rho};
$
\item[\itemno2]
$
\sup_{y\in M} \, \, \bignorm{k_{F(t\cD_1)}(\cdot,y)}{2} 
\leq C\, t^{-n/2}\, (1+t)^{n/2-\rho};
$
\item[\itemno3]
$
\sup_{y \in M} \, \bignorm{\dest_2 k_{F(t\cD_1)}(\cdot,y)}{2}
\leq C\, t^{-n/2-1}\, (1+t)^{n/2+1-\rho},
$
where $\dest_2$ denotes exterior differentiation with respect to
the second variable. 
\end{enumerate}
\end{proposition}

\begin{proof} 
We may assume that the kernels $k_{F(t\cD)}$ and $k_{F(t\cD_1)}$ 
are smooth. 
Indeed, it suffices to prove that  the desired estimates hold for  
all functions $G$  with bounded support such that 
$\mod{G}\le \mod{F}$, with a constant 
$C$  that does not depend on the support.  
Since for such functions the operator 
$\cL^N G(t\cD)$ is bounded on $L^2(M)$ for every positive integer 
$N$, its kernel is a smooth function on 
$M$, by elliptic regularity. The general case will follow by 
approximating~$F$ with functions of bounded support. 

First we prove \rmi. 
Suppose that $\sigma>n/4$.  Then, by (\ref{f: special}),
\begin{align}
\bigopnorm{(1+t^2\cD^2)^{-\sigma}}{1;2}
& \le    \smallfrac{1}{\Ga(\sigma)} \,\,  \Bigopnorm{{\ioty s^{\sigma} 
	 \, \e^{-s} \, \e^{-st^2\cD^2} {\dtt s}}}{1;2} \nonumber\\ 
& \leq C \int_0^{1/t^2} s^{\sigma} 
	  \, \e^{-s} \, (t^2s)^{-n/4} {\dtt s} \nonumber\\
&\qquad + C \int_{1/t^2}^\infty s^{\sigma} \,\e^{- s} 
	  \, (t^2s)^{-\rho/2}{\dtt s}\nonumber\\ 
& \leq C \, t^{-n/2}\, (1+t)^{n/2-\rho} 
	  \quant t \in \BR^+.\label{f: ultrac resolvent}
\end{align} 
By the spectral theorem and the assumption on $F$
$$
\sup_{t>0} \, 
\bigopnorm{(1+t^2\cD^2)^{\ga/2} \, F(t\cD)}{2} 
= \sup_{\la>0} \mod{(1+\la^2)^{\ga/2} \, F(\la)} 
< \infty. 
$$
Thus, by applying (\ref{f: ultrac resolvent}) with $\sigma
=\ga/2$, we get
\begin{align}
\bigopnorm{F(t\cD)}{1;2} 
& =    \bigopnorm{(1+t^2\cD^2)^{-\ga/2}
	\, (1+t^2\cD^2)^{\ga/2} \, F(t\cD)}{1;2} \nonumber \\
& \leq \bigopnorm{(1+t^2\cD^2)^{\ga/2} 
	\, F(t\cD)}{2}
	\, \bigopnorm{(1+t^2\cD^2)^{-\ga/2}}{1;2} \nonumber \\
& \leq C \,t^{-n/2}\, (1+t)^{n/2-\rho}
	  \quant t \in \BR^+.    \label{1;2}
\end{align}
Then the adjoint operator $F(t\cD)^*$ maps
$\ld{M}$ boundedly into $\ly{M}$ and
$$
\bigopnorm{F(t\cD)^*}{2;\infty} 
= \bigopnorm{F(t\cD)}{1;2}. 
$$
Thus, by Dunford-Pettis' Theorem \cite[Thm~6, p.~503]{DS},
the kernel $k_{F(t\cD)^*}$ of $F(t\cD)^*$ 
satisfies the estimate
\begin{align*}
\sup_{x\in M}\, \norm{k_{F(t\cD)^*}(x,\cdot)}{2}
& = \bigopnorm{F(t\cD)}{1;2} \\
& \leq  C \,t^{-n/2}\, (1+t)^{n/2-\rho}
	  \quant t \in \BR^+.
\end{align*}
Estimate \rmi\ follows from this and the fact that
$k_{F(t\cD)}(x,y) = \overline{k_{F(t\cD)^*}(y,x)}$.
\par
The proof of \rmii\ is, \emph{mutatis mutandis}, the same as 
the proof of \rmi.  We simply replace $\cD^2$ by $\cD_1^2$
in the proof of \rmi, and use the obvious ultracontractive bounds
for $\e^{-st^2\cD_1^2}$, instead of those for $\e^{-st^2\cD^2}$.

Finally we prove \rmiii.  By arguing as in the proof of \rmi\
(with $\bD_1$ in place of $\cD$, and using the ultracontractivity
estimates (\ref{ucLex}) in place of (\ref{f: special})),
we may show that there exists a constant $C$ such that 
\begin{equation} \label{m} 
\bigopnorm{F(t\bD_1)}{1;2}
\le C\,t^{-n/2} \, (1+t)^{n/2-\rho}
	\quant t \in \BR^+.
\end{equation}

We claim that there exists a constant $C$ such that
\begin{equation} \label{dm} 
\bigopnorm{\de\,F(t\bD_1)}{1;2}
\le C\,t^{-n/2-1}\, (1+t)^{n/2+1-\rho}
\quant t \in \BR^+.  
\end{equation}
To prove (\ref{dm}), observe that for every $\om$ in $C^\infty_c(\La^1_\BC M)$
\begin{align*}
\norm{\de\, F(t\bD_1)\,\om}{2}^{\hskip-5pt 2}
& =\langle\de \, F(t\bD_1)\,\om, \de\, F(t\bD_1)\,\om\rangle_0 \\ 
&\le \langle\de \, F(t\bD_1)\,\om, \de\, F(t\bD_1)\,\om\rangle_0 
    +\langle\dest \, F(t\bD_1)\,\om, \dest\, F(t\bD_1)\,\om\rangle_2 \\
& =\langle\bL \, F(t\bD_1)\,\om,  F(t\bD_1)\,\om\rangle_1 \\ 
& =\norm{ \bD_1\,F(t\bD_1) \om}{2}^{\hskip-5pt 2}
    +(b-\kappa^2) \, \norm{F(t\bD_1) \om}{2}^{\hskip-5pt 2}.
\end{align*}
This and (\ref{m}) imply that  
\begin{equation}\label{est1;2}
\begin{aligned}
\bigopnorm{\de\, F(t\bD_1)}{1;2}
& \le \bigopnorm{\bD_1\,F(t\bD_1)}{1;2}+b\,  \bigopnorm{F(t\bD_1)}{1;2} \\
& \le \bigopnorm{\bD_1\,F(t\bD_1)}{1;2} + C \, t^{-n/2}\, (1+t)^{n/2-\rho}
   \quant t \in \BR^+.
	\end{aligned}
\end{equation}
To conclude the proof of the claim it suffices to observe
that
\begin{align*}
\bigopnorm{t\bD_1\, F(t\bD_1)}{1;2} 
& =    \bigopnorm{(1+t^2\bD_1^2)^{-(\ga-1)/2}
	\, (1+t^2\bD_1^2)^{(\ga-1)/2} \, t\bD_1\, F(t\bD_1)}{1;2} \\
& \leq \bigopnorm{(1+t^2\bD_1^2)^{(\ga-1)/2} 
	\, t\bD_1\, F(t\bD_1)}{2}
	\, \bigopnorm{(1+t^2\bD_1^2)^{-(\ga-1)/2}}{1;2} \\
& \leq  C \, t^{-n/2}\, (1+t)^{n/2-\rho}
	\quant t \in \BR^+,
\end{align*}
where we have used (\ref{f: ultrac resolvent}) with $\sigma = (\ga-1)/2$.

Recall that $F(t\cD_1)\,\de =\de\,F(t\bD_1)$ by Lemma~\ref{l: propLex}~\rmi.
Thus, by (\ref{dm}), the operator $F(t\cD_1)\,\de$  
maps $L^1(\La^1_\BC M)$ to $L^2(M)$, and
$$
\bigopnorm{F(t\cD_1)\,\de}{1;2}
\le C\,t^{-n/2-1}\, (1+t)^{n/2+1-\rho} \quant t\in\BR^+.
$$
Hence its adjoint $\dest \,{F}(t\cD_1)^*$ 
maps $L^2(M)$ to $L^\infty(\La^1_\BC M)$ and 
$$
\bigopnorm{\dest\,{F}(t\cD_1)^*}{2;\infty}
\le C\,t^{-n/2-1}\, (1+t)^{n/2+1-\rho}.
$$
Thus, by Dunford-Pettis' Theorem, the kernel $k_{\dest F(t\cD_1)^*}$ 
of the operator $\dest\,{F}(t\cD_1)^*$ satisfies the estimate
$$
\sup_{y\in M}\, \bignorm{k_{\dest F(t\cD_1)^*}(y,\cdot)}{2} 
\le C\,t^{-n/2-1}\, (1+t)^{n/2+1-\rho}\quant t\in\BR^+.
$$
The desired conclusion follows because $\dest_2\,k_{F(t\cD_1)}(x,y)
= \overline{k_{\dest F(t\cD_1)^*}(y,x)}$.
\end{proof}

\subsection{Some technical lemmata} \label{subsec: Notation}

To motivate the technical result contained in this subsection, 
we briefly recall the main features of 
Taylor's method to prove spectral multiplier theorems for the 
Laplace--Beltrami operator on a Riemannian manifold $M$ of bounded geometry. 
Consider an operator of the form $m(\cD)$, where $m$ 
is an even, bounded, holomorphic function in the strip 
$\bS_{\be}$ and satisfies Mihlin-type conditions at infinity (see
(\ref{f: pseudodiff estimates 1})).
One of the main ingredients of Taylor's method is the  
functional calculus formula
\begin{equation}\label{f:taylor}
m(\cD)
= \frac{1}{2\pi} \,  \ir \wh {m}(t)\,\cos(t\cD)\wrt t,
\end{equation} 
based on the Fourier inversion formula and the spectral theorem.
The analysis of $m(\cD)$ ultimately relies on the finite propagation speed 
property for the wave equation and  uniform Sobolev estimates on 
$M$, proved in \cite{CGT} under rather strong bounded 
curvature assumptions on the manifold $M$. 
Since, in this paper, we want to relax the latter 
assumption by requiring only a lower bound on 
the Ricci curvature of $M$, we need to modify 
Taylor's proof.  
The aim of this section is to provide some of 
the required technical ingredients.

The first step consists in replacing the cosine in 
the right hand side of (\ref{f:taylor}) by a modified Bessel function
(see Lemma \ref{l: propertiesRiesz}~\rmii). 
For each $\nu\geq -1/2$, denote by $\cJ_\nu: \BR\setminus\{0\} \to \BC$
the modified Bessel function of order $\nu$, defined by
$$
\cJ_\nu(t) = \frac{J_\nu(t)}{t^\nu},
$$
where $J_\nu$ denotes the standard Bessel function of the first
kind and order $\nu$ (see, for instance, 
\cite[formula~(5.10.2), p.~114]{L}). We recall that, if $\Re \nu>-1/2$,
\begin{equation}\label{bessel}
J_\nu(t)=\frac{2^{1-\nu}}{\sqrt{\pi}\ \Gamma(\nu+1/2)}\ t^\nu\ \int_0^1(1-s^2)^{\nu-1/2}\,\cos(ts)\wrt s
\end{equation}
and that
$$
\cJ_{-1/2} (t)
= \sqrt{\frac{2}{\pi}}\  \cos t
\qquad \hbox{and}\qquad
\cJ_{1/2} (t)
= \sqrt{\frac{2}{\pi}}\ \frac{\sin t}{t}. 
$$

We recall the definition of the 
generalised Riesz means,  introduced in [CM, Section~1],
and summarise some of their properties.

Suppose that $d$ and $z$ are complex numbers such that 
$\Re d>0$ and that $\Re z >0.$  For every
$f$ in the Schwartz class $\cS(\BR)$, the
\emph{generalised Riesz mean of order} $(d,z)$ of $f$ is the 
function $R_{d,z}f,$ defined by
$$
R_{d,z}f (t) = \smallfrac{2}{\Ga(z)} \int_0^1 s^{d-1} \, 
(1-s^2)^{z-1} \, f(st) \wrt s
\quant t \in \BR.
$$
For fixed $d$ and $t$, the function $z\mapsto R_{d,z}f(t)$
has analytic continuation to an entire function. 
\par 
For every $f$ in $\lu{\BR}$ 
define its Fourier transform $\widehat f$ by
$$
\widehat f(t) = \ir f(s) \, e^{-ist} \wrt s
\quant t \in \BR.
$$
Sometimes we write $\cF f$ instead of $\wh f$, and denote by $\cF^{-1}f$
the inverse Fourier transform of~$f$.

Suppose that $f$ is a function on $\BR$, and that $\lambda$ is in $\BR^+$.
We denote by $f^\lambda$ and $f_\lambda$ the~$\lambda$-dilates of $f$, defined by
\begin{equation} \label{f: dilate} 
f^\lambda(x)
= f(\lambda x)  
\qquad\hbox{and}\qquad
f_\lambda(x)
= \lambda^{-1} \, f(x/\lambda)  \quant x \in \BR.
\end{equation}

For each positive integer $h$,
we denote  by $\cO^h$ the differential operator $t^h\, D^h$ on the real~line.

\begin{lemma}[{[CM]}] \label{l: propertiesRiesz old}
Suppose that $k$ is a positive integer, $d$ and $w$ are 
complex numbers, and $\Re d>0$.  The following hold:
\begin{enumerate}
\item[\itemno1]
if $\Re (d-2w)>0,$ then $R_{d-2w,w} \, R_{d,z} = R_{d-2w,w+z};$
\item[\itemno2]
$R_{d,0}$ is the identity operator;
\item[\itemno3]
there exist constants $C_{j,d,k}$ such that
$
R_{d,-k} = \sum_{j=0}^k C_{j,d,k} \, \cO^j.
$
\end{enumerate}
\end{lemma}

\begin{proof}
The proofs of \rmi, \rmii, and \rmiii\ may be found in \cite[Section~1]{CM}.
\end{proof}

We shall make repeated use of the operator $R_{1+2k,-k}$.
For notational convenience, in the rest of this paper we shall 
write $R_k$ instead of $R_{1+2k,-k}$, and we shall denote the formal 
adjoint of $R_k$ by $R_k^*$.  Thus
$$
\ir f(t) \, R_kg(t) \wrt t
= \ir R_k^*f(t) \, g(t) \wrt t
\quant f,g \in \cS(\BR).
$$

\begin{lemma}\label{l: propertiesRiesz}
Suppose that $k$ is a positive integer.  The following hold:
\begin{enumerate}
\item[\itemno1]
if $g$ is a bounded smooth function 
with bounded derivatives up to the order~$k$, and $f$ is 
rapidly decreasing together with its derivatives up to the order~$k$,
then there exist constants $C_{h,k}^*$ such that
$$
\prodo{f}{R_{k}g}
= \sum_{h=0}^k  C_{h,k}^* \ir \cO^h f(t) \, g(t) \wrt t.
$$ 
Thus $R^*_k=\sum_{h=0}^k C_{h,k}^*\mathcal O^h$.
\item[\itemno2]
if $f$ is a tempered distribution such that 
$\cO^h f$ is in $\lu{\BR}\cap C_0(\BR)$ for all
$h$ in $\{0,1,\ldots, k\}$, then 
$$
\ir {f} (t) \, \cos (vt) \wrt t
= \sqrt\pi \, 2^{k-1/2}  
     \ir  R_k^*{f}(t)\,  \cJ_{k-1/2} (t v)  \wrt t.
$$
\end{enumerate}
\end{lemma}

\begin{proof}
First we prove \rmi.  By using Lemma \ref{l: propertiesRiesz old} \rmiii, and then integrating by parts
$$
\begin{aligned}
\prodo{f}{R_{k} g}
& =  \sum_{j=0}^k  C_{j,1+2k,k} \ir f(t) \, \cO^j g(t) \wrt t \\
& =  \sum_{j=0}^k  C_{j,1+2k,k} (-1)^j 
		\ir D^j\bigl(t^j \, f\bigr)(t) \ g(t) \wrt t.
\end{aligned}
$$
Define 
$\ds
a_{j,\ell}
= {j\choose \ell}\frac{j!}{\ell!}.
$
By Leibniz's r\^ule
$
D^j\bigl(t^j \, f\bigr)(t)
=  \sum_{\ell=0}^j a_{j,\ell} \, \cO^{j-\ell} f (t).
$
Then
$$
\begin{aligned}
\prodo{f}{R_{k} g}
& =  \sum_{j=0}^k  C_{j,1+2k,k} (-1)^j 
	   \sum_{\ell=0}^j a_{j,\ell}
	   \ir \cO^{j-\ell}f (t) \, g(t) \wrt t \\ 
& =  \sum_{j=0}^k  C_{j,1+2k,k} (-1)^j 
	   \sum_{h=0}^j a_{j,j-h}
	   \ir \cO^{h}f (t) \, g(t) \wrt t \\ 
& =  \sum_{h=0}^k  C_{h,k}^* \ir \cO^h f(t) \ g(t) \wrt t, 
\end{aligned}
$$
where $C_{h,k}^* = \sum_{j=h}^k (-1)^j \,C_{j,1+2k,k}\, a_{j,j-h}$,
as required.

Next we prove \rmii.
For every $v$ in~$\BR^+$, denote by $C^v$ the function
$
C^v (t) = \cos(tv).
$
The required formula follows from \rmi, once we prove that 
$$
C^v  = \sqrt\pi \, 2^{k-1/2} \, R_{k}\bigl(\cJ_{k-1/2}^v).
$$
To prove this formula, observe that for every positive integer $k$
$$
\begin{aligned}
\bigl(R_{1,k} C^v\bigr)(t)
&  = \smallfrac{2}{\Ga(k)} \int_0^1 (1-s^2)^{k-1} \, \cos(stv) \wrt s \\
&  = \sqrt\pi \, 2^{k-1/2} \, \cJ_{k-1/2} (t v)
\end{aligned}
$$
by (\ref{bessel}).
Then we use Lemma~\ref{l: propertiesRiesz old}~\rmi\ and \rmii, and write
$$
\sqrt\pi \, 2^{k-1/2} \, R_{k}\bigl(\cJ_{k-1/2}^v)
= R_{k}\bigl(R_{1,k}C^v\bigr) = R_{1+2k,0}C^v = C^v,
$$
as required.
\end{proof}

In the rest of this section we shall provide various 
estimates of  functions of the form $R_k^*\wh{g}$ that, 
in combination with Lemma~\ref{l: propertiesRiesz}~\rmi, 
will be needed in the proof of Theorem~\ref{t: multiplier 2}.

Suppose that $J$ is a positive number.
Denote by $\vp:\BR \to [0,1]$ a smooth even function,
supported in $[-4,-1/4] \cup [1/4,4]$, 
equal to one on $[-2,-1/2] \cup [1/2,2]$, and such that
$\sum_{j\in \BZ} \vp^{2^{-j}} = 1$ on $\BR\setminus \{0\}$.
We denote by $H^J(\BR)$ the standard Sobolev space on $\BR,$ modelled
over $\ld\BR$.  

\begin{definition}
We say that a function $g:\BR\to \BC$ satisfies a
\emph{H\"ormander condition}~\cite{Ho} of order $J$ on the real line if 
\begin{equation} \label{f:Hormandercondition}
\sup_{\lambda>0} \norm{\vp\, g^\lambda}{H^J(\BR)} 
< \infty.
\end{equation}
We set $\norm{g}{\Horm(J)} := \sup_{\lambda>0} \norm{\vp\, g^\lambda}{H^J(\BR)}$.
\end{definition}

Note that (\ref{f:Hormandercondition}) implies that 
$\norm{g}{\infty} \le 2\, \norm{g}{\Horm(J)}$ if $J>1/2$.
We need a technical lemma, which is a version of H\"ormander's method \cite{Ho}.

\begin{lemma} \label{l: technical}
Suppose that $s$ is in $(1/2,\infty)$, that $k$ is a positive integer
and that $g: \BR\to\BC$ is a bounded even function that extends to an 
entire function of exponential type $1$.
For each integer~$j$ define the functions $g_j$ by
$$
g_j = g\, \vp^{2^{-j}}.
$$
Then $\wh{g_j}$ is an entire function of exponential type and 
$\wh g = \sum_j \wh g_j$ in the sense of distributions.
Furthermore, for every $\vep$ in $[0,s-1/2)$
there exists a constant $C$ such that for all $r$ in $\BR^+$
\begin{enumerate}
\item[\itemno1]
$\ds
\int_{\mod{t}>r} \mod{R_k^*\wh{g_j}(t)} \wrt t 
\leq  C \, (2^j\, r)^{-\vep}\, \norm{g}{\Horm(s+k)}
$;
\item[\itemno2]
$ \ds 
r\, \int_{\mod{t} > r} \frac{\mod{R_k^*\wh{g_j}(t)}}{\mod{t}} \wrt t 
\leq C \,  (2^j\, r)^{1/2} \, \norm{g}{\Horm(k)}$;
\vskip5truept
\item[\itemno3]
$ \ds 
\sup_{j\in \BZ} \, \norm{\wh{g_j}}{1} 
\leq C \, \norm{g}{\Horm(s)}.
$
\end{enumerate}
\end{lemma}

\begin{proof}
For every integer $\ell$ in $\{0,\ldots,k\}$
define the tempered distribution $G^\ell$ and the functions $G_j^\ell$ by
$$
G^\ell = \cO^\ell \wh g,
\qquad
\qquad \hbox{and}\qquad
G_j^\ell = \cO^\ell \wh{g_j}.
$$
By Lemma \ref{l: propertiesRiesz} \rmi, $R_k^*\wh{g_j} = \sum_{\ell=0}^k C_{\ell,k}^* \, \cO^\ell\wh{g_j}$.
Thus, to prove \rmi\ and \rmii\ it suffices to prove similar estimates
with $G_j^\ell$ in place of $R_k^*\wh{g_j}$ for all $\ell$
in $\{0,\ldots, k\}$.
\par
Note that both $\wh{g_j}$ and $G_j^\ell$ are entire functions 
of exponential type $2^{j+2}$.  Observe that
$$
\wh{g_j} = \cF  \bigl[(g^{2^j} \vp)^{2^{-j}} \bigr]
= \bigl[\cF  (g^{2^j} \vp)\bigr]_{2^{-j}}\ .
$$
Hence
\begin{equation}
G_j^\ell 
 = \cO^\ell \bigl\{ \bigl[\cF  (g^{2^j} \vp)\bigr]_{2^{-j}}\bigr\}
 = \bigl[\cO^\ell \cF  (g^{2^j} \vp)\bigr]_{2^{-j}}.
\end{equation}
By elementary Fourier analysis
$
\cF^{-1} \bigl[\cO^\ell\, \cF (g^{2^j}\, \vp)\bigr] (\xi)
= (-1)^{\ell} \, D^\ell\bigl[ \xi^\ell \,g^{2^j} \,\vp \bigr](\xi).
$

Now we prove \rmi.  Note that
\begin{equation}
\begin{aligned} \label{f: real est I}
\int_{\mod{t}>r} \mod{G_j^\ell(t)} \wrt t 
& =    \int_{\mod{t}>r} \bigmod{\bigl[\cO^\ell\cF (g^{2^j} \vp)
	\bigr]_{2^{-j}}(t)} \wrt t \\
& =    \int_{\mod{t}>2^j\, r} \bigmod{\cO^\ell\cF (g^{2^j} \vp)(t)} \wrt t \\
& \leq   (2^j\, r)^{-\vep}\,
 \int_{\BR} \mod{t}^{\vep}\, \bigmod{\cO^\ell\cF (g^{2^j} \vp)(t)} \wrt t \\
& \leq C\,   (2^j\, r)^{-\vep}\, \bignorm{\cF^{-1} \cO^\ell\cF (g^{2^j} \,
	 \vp)}{H^s(\BR)}
\end{aligned}
\end{equation}
by the classical Bernstein's Theorem,
where $\vep$ is in $[0,s-1/2)$.  Note that $C$ depends on $\vep$
but is independent of $j$.
Now, by Plancherel's Theorem,
$$
\begin{aligned}
\norm{\cF^{-1} \bigl[\cO^\ell\, \cF (g^{2^j}\, \vp)\bigr]}{H^s(\BR)}\raise5pt\hbox{\hskip-25pt$^2$}\hskip25pt
& =    \int_{\BR} \bigmod{\cO^\ell\cF \bigl[g^{2^j} \,\vp \bigr](t)}^2
       \, \bigl(1+\mod{t}^2\bigr)^s \wrt t \\
& \leq \int_{\BR} \bigmod{D^\ell\cF \bigl[g^{2^j} \,\vp \bigr](t)}^2
       \, \bigl(1+\mod{t}^2\bigr)^{s+\ell} \wrt t \\
& =    \int_{\BR} \bigmod{\cF \bigl[\xi^\ell \, g^{2^j} \,\vp \bigr](t)}^2
       \, \bigl(1+\mod{t}^2\bigr)^{s+\ell} \wrt t. 
\end{aligned}
$$
The square root of the last integral is a constant times 
$\norm{\xi^\ell \,g^{2^j} \,\vp}{H^{s+\ell}(\BR)}$, which is clearly dominated
by  $\norm{g^{2^j} \,\vp}{H^{s+\ell}(\BR)}$.
Then (\ref{f: real est I}) implies that
\begin{equation}
\begin{aligned} \label{f: real est III}
\int_{\mod{t}>r} \mod{G_j^\ell(t)} \wrt t 
& \leq  C \, (2^j\, r)^{-\vep}\, \norm{g}{\Horm(s+\ell)},
\end{aligned}
\end{equation}
as required to conclude the proof of \rmi.
\par
Next we prove \rmii.
Observe that
\begin{equation}
\begin{aligned}
r\, \int_{\mod{t} > r} \frac{\mod{G_j^\ell(t)}}{\mod{t}} \wrt t 
& =    r\, \int_{\mod{t} > r} \frac{\bigmod{\bigl[\cO^\ell \cF (g^{2^j} \,
       \vp)\bigr]_{2^{-j}}(t)}}{\mod{t}} \wrt t \\
& =   2^{j}\,  r\, \int_{\mod{t} > 2^j\,r} \frac{\bigmod{\cO^\ell 
       \cF (g^{2^j} \, \vp)(t)}}{\mod{t}} \wrt t \\
& \leq  2^{j}\,  r\, \Bigl(\int_{\mod{t} > 2^j\,r} 
       \mod{t}^{-2} \wrt t\Bigr)^{1/2} \,
       \bignorm{\cO^\ell \cF (g^{2^j} \, \vp)}{2} \\
& \leq C \,  \bigl(2^{j}\, r\bigr)^{1/2} \, \norm{g^{2^j} \, \vp}{H^\ell(\BR)} \\
& \leq C \,  (2^j\, r)^{1/2} \, \norm{g}{\Horm(\ell)},
\end{aligned}
\end{equation}
as required.\par
The inequality \rmiii \ follows from \rmi \  by taking $k=\vep=0$.
\end{proof}

\begin{remark} \label{rem: technical} 
Notice the following variant of Lemma~\ref{l: technical}~\rmii\
that will be used in the proof of Theorem~\ref{t: multiplier 2}~\rmi\ below.  
For every $\eta$ in $(1/2,1]$ and for every $R$ in $\BR^+$
there exists a constant $C$ such that 
\begin{equation} \label{f: ultima}
r\, \int_{\mod{t} > r} \frac{\mod{R_k^*\wh{g_j}(t)}}{\mod{t}^{\eta}} \wrt t 
\leq C \,  (2^j\, r)^{1/2} \, \norm{g}{\Horm(k)}
\quant r \in (0,R].
\end{equation}
The proof is much the same as the proof of Lemma \ref{l: technical}  \rmii. 
As before, it suffices to prove (\ref{f: ultima}) with $G_j^\ell$ in place 
of $R_k^*\wh{g_j}$ for all $\ell$ in $\{0,\ldots, k\}$.

Observe that
\begin{equation}
\begin{aligned}
r\, \int_{\mod{t} > r} \frac{\mod{G_j^\ell(t)}}{\mod{t}^{\eta}} \wrt t 
& =    r\, \int_{\mod{t} > r} \frac{\bigmod{\bigl[\cO^\ell \cF (g^{2^j} \,
       \vp)\bigr]_{2^{-j}}(t)}}{\mod{t}^{\eta}} \wrt t \\
& =   2^{\eta j}\,  r\, \int_{\mod{t} > 2^j\,r} \frac{\bigmod{\cO^\ell 
       \cF (g^{2^j} \, \vp)(t)}}{\mod{t}^{\eta}} \wrt t \\
& \leq  2^{\eta j}\,  r\, \Bigl(\int_{\mod{t} > 2^j\,r} 
       \mod{t}^{-2\eta} \wrt t\Bigr)^{1/2} \,
       \bignorm{\cO^\ell \cF (g^{2^j} \, \vp)}{2} \\
& \leq C \,  2^{j/2}\, r^{3/2-\eta} \, \norm{g^{2^j} \, \vp}{H^\ell(\BR)} \\
& \leq C \,  (2^j\, r)^{1/2} \, \norm{g}{\Horm(\ell)},
\end{aligned}
\end{equation}
as required.  Note that in the last inequality
we have used the fact that $r$ varies in a bounded~set.
\end{remark}

\section{Spectral multipliers on Riemannian manifolds}  
\label{s: Spectral multipliers on Riemannian manifolds}

In this section we prove our main result, Theorem~\ref{t: multiplier 2}.
To treat the part of the kernel $k_{m(\cD)}$ near the diagonal of $M\times M$,
we shall need the following result, which 
is the analogue on manifolds with bounded geometry of a 
well known result in the setting of spaces of homogeneous 
type in the sense of Coifman and Weiss \cite{CW}.
For the reader's convenience we sketch its proof, but omit
the details of the part which is 
very similar to the proof of \cite[Th\'eor\`eme~2.4]{CW}. 

We denote by $\cB_s$ the family 
of all balls with radius at most $s$. Given a ball $B$, we denote by $c_B$ its centre, by $r_B$
its radius and by $2B$ the ball with the same centre as $B$ and radius $2r_B$.    

\begin{proposition}  \label{rem: hom type}
Suppose that $\cT$ is a bounded operator on $\ld{M}$ 
and that
\begin{enumerate}
\item[\itemno1]
its Schwartz kernel $k_{\cT}$ is locally integrable 
in $(M\times M) \setminus \{(x,x): x \in M\}$,
and is supported in $\{(x,y)\in M\times M:
d(x,y) \leq 1\}$;
\item[\itemno2]
the following \emph{H\"ormander integral condition at scale $1$} holds
$$
A:=
\sup_{B \in \cB_1} \, \sup_{y\in B}  \, \, \int_{B(c_B,2) 
\setminus (2B)}
\mod{k_{\cT}(x,y) -  k_{\cT}(x,c_B)} \wrt \mu(x) 
< \infty.
$$
\end{enumerate}
Then $\cT$ extends to an operator of weak type $1$ and there exists a constant $C$ such that 
$$
\opnorm{\cT}{L^1(M);L^{1,\infty}(M)}\le\,CA.
$$
\end{proposition}

\begin{proof}
Denote by $\fM$ a $1$-discretisation of $M$, i.e., a subset of $M$
that is maximal with respect to the following property:
$$
d(z_1,z_2) \geq 1 \quant z_1,z_2 \in \fM
\qquad\hbox{and} \qquad
d(x, \fM) \leq 1 \quant x \in M.
$$
We denote by $\{z_j: j \in \BN\}$
the points of $\fM$.  Since the measure $\mu$ is locally doubling,  the family $\{B(z_j,1): z_j \in \fM\}$ is a covering
of~$M$ such that $\{B(z_j,2): z_j \in \fM\}$ has the \emph{bounded overlap property}, i.e., 
there exists a positive integer $N_2$ such that 
$$
1\leq \sum_{j\in \BN} \One_{B(z_j,1)} \leq \sum_{j\in \BN} \One_{B(z_j,2)} \leq N_2,
$$
where $\One_E$ denotes the indicator function of the set $E$.
Given $f$ in $\lu{M}$ and a nonnegative integer $j$, we define $f_j$ by
$
f_j = f\, \One_{B(z_j,1)}/\sum_\ell \One_{B(z_\ell,1)}.
$
Then $f = \sum_{j\in \BN} f_j$, and 
$$
\cT f = \sum_{j\in \BN} \cT f_j.
$$
Note that this sum is locally uniformly finite, because the function $\cT f_j$ 
is supported in the ball~$B(z_j,2)$, by \rmi\ above, and 
the family $\{B(z_j,2): z_j \in \fM\}$ has the bounded overlap property. Then there exists a constant $C$ such that 
$$
\mu\bigl( \bigl\{x \in M: \bigmod{\cT f(x)} > s \bigr\} \bigr)
\leq C \, \sum_{j\in \BN} 
\mu\bigl( \bigl\{x \in M: \bigmod{\cT f_j(x)} > s/N_2 \bigr\} \bigr)
\quant s \in \BR^+.
$$
Thus, to conclude the proof it suffices to show that there exists a constant
$C$ such that 
\begin{equation} \label{f:  sum}
s\ \mu\bigl( \bigl\{x \in M: \bigmod{\cT f_j(x)} > s \bigr\} \bigr)
\leq C \, A \, {\norm{f_j}{1}}  \quant s \in \BR^+ \quant j \in \BN,
\end{equation}
for then we may conclude that 
$$
\begin{aligned}
s\ \mu\bigl( \bigl\{x \in M: \bigmod{\cT f(x)} > s \bigr\} \bigr)
& \leq C \, A\, \sum_{j\in \BN} \,  \norm{f_j}{1} \\
& \leq C \,  A\, \norm{f}{1} \qquad\forall s\in\BR^+.
\end{aligned}
$$
by the bounded overlap property, as required.   

To prove (\ref{f: sum}), we may follow the proof of the original
result of R.R. Coifman and G.~Weiss on spaces of 
homogeneous type \cite[Th\'eor\`eme~2.4]{CW}.  Define the 
\emph{local doubling constant} $D_2$~by 
$$
D_2 
= \sup_{B\in \cB_2}\, \frac{\mu(2B)}{\mu{(B)}}.
$$ 
Then, given $s$ in $\BR^+$, consider a Calder\'on--Zygmund
decomposition of $f_j$ at height $s$.  Note that, though $M$
need not be a space of homogeneous type, each $f_j$ is 
supported in a ball of radius $1$, and all these balls
are spaces of homogeneous type with doubling constant dominated by $D_2$.    
Thus, the constants appearing in the Calder\'on--Zygmund decompositions
of the functions $f_j$ depend on $D_2$, but not on $j$.  
Then the proof of (\ref{f: sum}) is 
exactly as in the setting of spaces of homogeneous type,
and the constant $C$ in (\ref{f: sum}) 
depends on $D_2$, but not on $s$ or $j$.  We omit the details. 
\end{proof}

Now we define an appropriate function space of holomorphic functions
which will be needed in the statement of Theorem~\ref{t: multiplier 2}.  
Then, for the reader's convenience, we recall 
one of its properties, which will be key in the proof of our main result. 

\begin{definition}  \label{d: Hormander at infinity}
Suppose that $J$ is a positive integer and that $W$
is in $\BR^+$.
We recall that $\bS_{W}$ denotes the strip $\{ \zeta \in \BC: \Im (\zeta)
\in (-W,W)\}$ and we denote by $H^\infty (\bS_{W};J)$ the vector space of 
all bounded \emph{even} holomorphic functions~$f$ in $\bS_{W}$ for which
there exists a positive constant $C$ such that
\begin{equation} \label{f: SsigmaJ}
\bigmod{D^j f(\zeta)} 
\leq {C} \, {(1+\mod{\zeta})^{-j}}
\quant \zeta\in \bS_{W} \quant j \in \{0,1,\ldots,J\}.
\end{equation}
The infimum of all constants $C$ for which (\ref{f: SsigmaJ}) holds
will be denoted by~$\norm{f}{\bS_{W};J}$.
\end{definition}

\begin{lemma}[{\cite[Lemma~5.4]{HMM}}] \label{l:splitf}
Suppose that $J$ is an integer $\geq 2$, and that $W$ is in $\BR^+$.
Then there exists a positive constant $C$ such that 
for every function $f$ in $H^\infty\bigl(\bS_{W};J\bigr)$, and
for every positive integer $h \leq J-2$
$$
\bigmod{\cO^h\widehat f (t)} 
\leq C\, \norm{f}{\bS_{W};J}\,  \mod{t}^{h-J} \, \e^{-W \mod{t}}
\quant t \in \BR\setminus\{0\}.
$$
\end{lemma}

\begin{theorem} \label{t: multiplier 2}
Suppose that $M$ is a Riemannian manifold with bounded geometry,
and suppose that (\ref{f: assumption Ht}) holds for some $\rho >1/2$.
Assume that $\al$ and~$\be$ are as in (\ref{f: volume growth}),
and denote by $N$ the integer $[\!\![n/2+1]\!\!] + 1$.
Suppose that $J$ is an integer $> \max(N + 1, N+1+\al/2-\rho)$.
Then there exists a constant~$C$ such that
$$
\opnorm{m(\cD)}{\lu{M};\lorentz{1}{\infty}{M}} 
\leq C \,  \norm{m}{\bS_{\be;J}}
\quant m \in H^\infty (\bS_{\be};J).
$$
\end{theorem}

\begin{proof}
For notational convenience in this proof we shall write $\cJ$
instead of $\cJ_{N-1/2}$.

Denote by $\om$ an even function in 
$C_c^\infty(\BR)$ which is supported in $[-1,1]$, is equal to~1
in $[-1/4,1/4]$, and satisfies 
$$
\sum_{j\in \BZ} \om(t-j) = 1
\quant t \in \BR.
$$
Clearly $\wh\om\ast m$ and $m-\wh\om\ast m$ are bounded functions.
We follow the strategy of Taylor
(see \cite[Thm~1]{T}), and 
define the operators $\cA $ and $\cB $ spectrally~by
$$
\cA 
= (\wh\om\ast m) (\cD)
\qquad\hbox{and}\qquad
\cB 
= (m-\wh\om\ast m) (\cD).
$$
Then
$
m(\cD) 
= \cA  + \cB .
$
We shall prove that there exists a constant $C$ such that 
\begin{equation}\label{f: M1}
\opnorm{\cA }{L^1(M);L^{1,\infty}(M)}\le\, C\, \norm{m}{\bS_\beta;J}
\end{equation}
and
\begin{equation}\label{f: M2}
\opnorm{\cB }{L^1(M)}\le\, C\, \norm{m}{\bS_\beta;J}.
\end{equation}
These estimates clearly imply the desired conclusion.

First we analyse the operator $\cA $.  Since $\wh{\om}\ast m$
is an even entire function of exponential type $1$, the function
$A$, defined by 
$$
A (\zeta)
= (\wh{\om}\ast m)\bigl(\sqrt{\zeta^2-\kappa^2}\bigr) 
\quant \zeta \in \BC,
$$
is entire of exponential type $1$.  
The reason for introducing the new function~$A$ is to write~$\cA $ 
as a function of the operator $\cD_1$ (defined at the beginning
of Subsection~\ref{subsec: Kernel}) rather than of
the operator~$\cD$.  Observe~that 
$$
\cA  = A (\cD_1),
$$
and that the support of $k_{\cA }$ is contained in 
$\{(x,y)\in M\times M: d(x,y) \leq 1\}$.
It is straightforward to check that 
\begin{equation} \label{f: rel Horm}
\bignorm{A}{\Horm(J)}
\leq C\, \bignorm{\wh{\om}\ast m}{\Horm(J)},
\end{equation}
where the constant $C$ does not depend on $m$.  
By arguing much as in the proof of \cite[Proposition~5.3]{HMM},
we may show that the function~$\wh{\om}\ast m$ satisfies a 
Mihlin--H\"ormander condition of order $J$, with $\norm{\wh{\om}\ast m}{\Horm(J)}$ bounded by a constant times $\norm{m}{\Horm(J)}$.  Furthermore,
it is clear that 
$$
\norm{m}{\Horm(J)}
\leq C\, \norm{m}{\bS_{\be};J},
$$
with $C$ independent of $m$.
In view of this observation and of Proposition~\ref{rem: hom type},
to prove that $\cA $ is of weak type $1$,
with the required norm estimate,
it suffices prove that its integral kernel $k_{\cA }$ satisfies
the following
\begin{equation} \label{f: LHIC}
\sup_{y\in B \in \cB_1} \, \int_{B(c_B,2) \setminus (2B)}
\mod{k_{\cA }(x,y) -  k_{\cA }(x,c_B)} \wrt \mu(x) 
\leq C\, \norm{A}{\Horm(J)}.
\end{equation}\par
To prove (\ref{f: LHIC})  
we further decompose the function $A$, and then decompose the 
operator~$\cA $ accordingly.
For all $j$ in $\BZ$
define the functions $A_j$ by
$$
A_j = A\, \vp^{2^{-j}},
$$
where $\vp$ is defined just above Lemma~\ref{l: technical}.
Then $\widehat{A}_j$ is an entire function of exponential type and 
$\widehat{A} = \sum_j \widehat{A}_j$ in the sense of distributions.
Furthermore, Lemma~\ref{l: technical} (with $A$ 
in place of $g$) and Remark~\ref{rem: technical} imply that  
for every $\eta$ in $(1/2,1]$
there exists a constant~$C$ such that for every $j$ in $\BZ$ and
for every $\ell$ in $\{0,1,\ldots,N\}$ \begin{align}
\int_{\mod{t}>r} \bigmod{R_N^* \wh{A}_j(t)} \wrt t 
&\leq  C\, \norm{A}{\Horm(J)} \, (2^j\, r)^{-\vep}
\qquad\forall r\in \BR^+\ ,\label{(a)} \\
r\, \int_{\mod{t} > r} \frac{\bigmod{R_N^* \wh{A}_j(t)}}{\mod{t}^\eta} \wrt t 
&\leq C  \, \norm{A}{\Horm(J)}\,  (2^j\, r)^{1/2}
\qquad\forall r\in(0,2].
 \label{(b)}
\end{align}
Here we have used the fact that $J>N+1/2$.

For each ball $B$ in $\cB_1$ and for each integer $j$, define $I_j(B)$ by 
$$
I_j(B)
= \sup_{y\in B} \int_{E_B} 
      \bigmod{k_{A_j(\cD_1)}(x,y)
       - k_{A_j(\cD_1)}(x,c_B)} \wrt \mu(x),
$$
where, for notational convenience, we write $E_B$ instead of 
$B(c_B,2) \setminus (2B)$.
To prove (\ref{f: LHIC}), it suffices to show that 
there exists a constant~$C$ such that for
$$
I_j(B)
\leq C\, \norm{A}{\Horm(J)}\, 
     \min \bigl((2^j \, r_B)^{-\vep}, (2^j \, r_B)^{1/2}\bigr)
\quant B \in \cB_1 \quant j \in \BZ.
$$
To prove this, we shall prove separately that
\begin{equation} \label{f: LHIC V}
I_j(B)
\leq C\, \norm{A}{\Horm(J)}\, (2^j \, r_B)^{-\vep},
\end{equation}
and that 
\begin{equation} \label{f: LHIC VI}
I_j(B)
\leq C\, \norm{A}{\Horm(J)}\, (2^j \, r_B)^{1/2}.
\end{equation}
The key formula here
is 
$$
A_j (\la)
= \frac{2^{N-1}}{\sqrt{2\pi}} \,
\ir R_N^* \wh{A}_j(t) \,\,  \cJ (t \la) \wrt t,
$$
which follows from the Fourier inversion formula and 
Lemma~\ref{l: propertiesRiesz}~\rmii, and its consequence
$$
k_{A_j(\cD_1)}
= \frac{2^{N-1}}{\sqrt{2\pi}} \,
\ir R_N^*\wh{A}_j(t) \,\,  k_{\cJ (t \cD_1)} \wrt t.
$$
Note that the modified Bessel function $\lambda\mapsto \cJ (t\lambda)$ is an even entire function of exponential type $\mod{t}$. 
Thus, by the 
finite propagation speed property for $\cL$, the kernel $k_{\cJ (t\cD_1)}(\cdot,y)$ vanishes outside the ball $B(y,\mod{t})$.\par
To prove (\ref{f: LHIC V}), note that 
$$
\begin{aligned}
I_j(B)
& \leq 2 \, \sup_{y\in B} \,  
  \bignorm{k_{A_j(\cD_1)}(\cdot,y)}{\lu{E_B}} \\
& \leq C \, \sup_{y\in B} \, \int_{\mod{t} \geq r_B}
      \bigmod{R_N^*\wh{A}_j(t)} \, \bignorm{k_{\cJ (t\cD_1)}
      (\cdot,y)}{\lu{E_B}} \wrt t.
\end{aligned}
$$
Now we split each of these integrals
as the sum of the corresponding 
integrals over the sets $\{t \in \BR: r_B\leq \mod{t} \leq 1\}$
and $\{t \in \BR: \mod{t} > 1\}$.  We denote these two integrals
by $\Upsilon_1$ and $\Upsilon_2$ respectively.  They depend on 
$y$ in $B$ and $j$.
\par
By the asymptotics of Bessel functions \cite[formula~(5.11.6), p.122]{L} 
$$
\sup_{s>0} \mod{(1+s)^{N} \, \cJ  (s)} < \infty,
$$
so that $\cJ$ satisfies the assumptions of 
Proposition~\ref{p:kernest} (with $N$ in place of $\ga$).
Hence by Schwarz's inequality, (\ref{ubc}) and Proposition~\ref{p:kernest}~\rmii\ 
$$
\begin{aligned}
\bignorm{k_{\cJ (t\cD_1)}(\cdot,y)}{\lu{B(c_B,2)}}
& \leq \mu\bigl(B(y,\mod{t})\bigr)^{1/2} \, 
    \bignorm{k_{\cJ (t\cD_1)}(\cdot,y)}{\ld{B(y,\mod{t})}} \\
& \leq C  \quant t \in [-1,1]\setminus\set{0}.
\end{aligned}
$$
As a consequence
\begin{align*}
\Upsilon_1&  \leq C \, \int_{r_B\leq \mod{t} \leq 1} \bigmod{R_N^*\wh{A}_j(t)}  \wrt t  \\ 
&\leq C  \, \norm{A}{\Horm(J)}\, (2^j\, r_B)^{-\vep}. \\ 
\end{align*}
Note that we have used (\ref{(a)}) above in the last inequality.
\par
To estimate $\Upsilon_2$ we argue similarly, 
using Proposition~\ref{p:kernest}~\rmi\  and the fact that $\mu(E_B)\le \mu\big(B(y,3)\big)\le C$ by (\ref{ubc}). Thus we obtain
\begin{align*}
\Upsilon_2&  \leq C \, \int_{\mod{t} > 1} \bigmod{R_N^*\wh{A}_j(t)} 
     \mod{t}^{-\rho}  \wrt t  \\ 
&\leq C \, \int_{\mod{t} \geq r_B} \bigmod{R_N^*\wh{A}_j(t)} \wrt t   \\
& \leq C  \, \norm{A}{\Horm(J)}\, (2^j\, r_B)^{-\vep}.
\end{align*}
Then (\ref{f: LHIC V}) follows.

To prove (\ref{f: LHIC VI}), observe that 
\begin{equation} \label{f: second estimate I}
\begin{aligned}
I_j(B)
& \leq C\, r_B \, \sup_{y\in B} \int_{E_B} 
      \bigmod{\dest_2 k_{A_j(\cD_1)}(x,y)} \wrt \mu(x) \\
& \leq C \, r_B \, \sup_{y\in B} \, \int_{\mod{t} \geq r_B}
      \bigmod{R_N^*\wh{A}_j(t)} \, \bignorm{\dest_2 
      k_{\cJ (t\cD_1)}(\cdot,y)}{\lu{E_B}} \wrt t.
\end{aligned}
\end{equation}
Much as before, we split each of these integrals
as the sum of the integrals over the sets $\{t \in \BR: r_B
\leq\mod{t} < 1\}$
and $\{t \in \BR: \mod{t} \geq 1\}$, and denote them by $\wt\Upsilon_1$
and $\wt\Upsilon_2$.

By the finite propagation speed for $\opL$, the kernel
${\dest_2 k_{\cJ (t\cD_1)}(\cdot,y)}$ 
vanishes outside the ball $B(y,\mod{t})$.  Hence
by Schwarz's inequality and Proposition~\ref{p:kernest}~\rmiii\ 
$$
\begin{aligned}
\bignorm{{\dest_2 k_{\cJ (t\cD_1)}(\cdot,y)}}{\lu{B(c_B,2)}}
& \leq \mu\bigl(B(y,\mod{t})\bigr)^{1/2} \, 
    \bignorm{{\dest_2 k_{\cJ (t\cD_1)}(\cdot,y)}}{\ld{B(y,\mod{t})}} \\
& \leq C \, \mod{t}^{n/2} \, \mod{t}^{-n/2-1} \quant t \in [-1,1]\setminus\set{0}
\end{aligned}
$$
and
$$
\begin{aligned}
\bignorm{{\dest_2 k_{\cJ (t\cD_1)}(\cdot,y)}}{\lu{B(c_B,2)}}
& \leq \mu\bigl(B(c_B,2)\bigr)^{1/2} \, 
    \bignorm{{\dest_2 k_{\cJ (t\cD_1)}(\cdot,y)}}{2} \\
& \leq C \, \mod{t}^{-\rho} \quant t \in \BR\setminus[-1,1].
\end{aligned}
$$
Then, by (\ref{(b)}),
$$
\wt\Upsilon_1
  \leq C \, \int_{r_B \leq\mod{t} \leq 1} \frac{\bigmod{R_N^*\wh{A}_j(t)}}{\mod{t}}
     \wrt t  
  \leq C  \, \norm{A}{\Horm(J)}\, (2^j/ r_B)^{1/2},
$$
and
$$
\wt\Upsilon_2
  \leq C \, \int_{\mod{t} > 1} \bigmod{R_N^*\wh{A}_j(t)} 
     \mod{t}^{-\rho}  \wrt t  
  \leq C  \, \norm{A}{\Horm(J)}\, (2^j/ r_B)^{1/2}.
$$
The required estimate (\ref{f: LHIC VI}) follows
from this and (\ref{f: second estimate I}).
This concludes the proof of (\ref{f: LHIC}),  hence of (\ref{f: M1}).
\par
Next we estimate $\opnorm{\cB }{L^1(M)}$.
For each $j$ in $\{2,3,\ldots,\}$, define $\om_j$ by the formula
$$
\om_j(t) = \om (t-j+1) + \om(t+j-1) \quant t \in \BR.
$$
Observe that $\cF (m-\wh\om\ast m) = \sum_{j=2}^\infty \om_j \, \wh m$.  Since
$m$ is in $H^\infty\bigl(\bS_{\be};J\bigr)$ and $J \geq N+2$, 
the function $\wh m$ and its derivatives up to the order $N$ are 
rapidly decreasing at infinity 
by Lemma~\ref{l:splitf}, so that 
$\cO^h\bigl((1-\om)\, \wh m\bigr)$ is in $\lu{\BR}\cap C_0(\BR^+)$ 
for all $h$ in $\{0,\ldots, N\}$.
Hence we may use 
Lemma~\ref{l: propertiesRiesz}~\rmii\ and write 
$$
\begin{aligned}
(m-\wh\om\ast m)(\la)
& = \frac{1}{2\pi} \, \sum_{j=2}^\infty \ir  \om_j(t) \, \wh{m}(t)
      \, \cos (t\la) \wrt t \\ 
& = \frac{2^{N-1}}{\sqrt{2\pi}} \,\sum_{j=2}^\infty   
  \ir  R_N^*\bigl(\om_j\, \wh m\bigr)(t)\, \cJ  (t \la)  \wrt t,
\end{aligned}
$$
for all $\lambda$ in $\BR$.  Define the kernel 
$k_{\cB }^j$ by 
$$
k_{\cB }^j 
= \frac{2^{N-1}}{\sqrt{2\pi}} \,
  \ir  R_N^*\bigl(\om_j\, \wh m\bigr)(t)\, k_{\cJ  (t \cD)}   \wrt t.
$$
Then, at least formally,
$
k_{\cB } 
=  \sum_{j=2}^\infty  k_{\cB }^j.
$
Note that 
$k_{\cB }^j$ is supported in 
$\{(x,y) \in M\times M: d(x,y) \leq j\}$ by finite propagation speed.
For all positive integer $\ell$ 
and for each $p$ in $M$, denote by $A(p,\ell)$ the annulus
with centre~$p$ and radii $\ell-1$ and $\ell$. 
Fix $y$ in $M$.  Then
$$
\begin{aligned}
\bignorm{k_{\cB } (\cdot,y)}{1}
&  =    \sum_{\ell = 1}^\infty \int_{A(y,\ell)} \bigmod{k_{\cB } (x,y)} 
     \wrt \mu(x)\\
&  \leq \sum_{\ell = 1}^\infty \mu\bigl(B(y,\ell)\bigr)^{1/2} \,
      \Bigl[\int_{A(y,\ell)} \bigmod{k_{\cB } (x,y)}^2  \wrt \mu(x)
      \Bigr]^{1/2}\\
\end{aligned}
$$
Note that if $j \leq \ell-1$, then the restriction of $k_{\cB }^j$
to $A(y,\ell)$ vanishes, because $k_{\cB }^j(\cdot,y)$ is supported in the 
ball $B(y,j)$.  
Thus, by Schwarz's inequality 
$$
\begin{aligned}
\bignorm{k_{\cB } (\cdot,y)}{1}
&  \leq \sum_{\ell = 1}^\infty \mu\bigl(B(y,\ell)\bigr)^{1/2} \,
      \sum_{j=\ell}^\infty \bignorm{k_{\cB }^j (\cdot,y)}{2}\\
&  \leq C\, \sum_{\ell = 1}^\infty\ell^{\al/2} 
      \, \e^{\be \ell} \, \sum_{j=\ell}^\infty 
      \ir \bigmod{R_N^*\bigl(\om_j\, \wh m\bigr)(t)} 
      \, \bignorm{k_{\cJ  (t \cD)}(\cdot,y)}{2}   \wrt t,
\end{aligned}
$$
where we have used (\ref{f: volume growth}) and the formula above for $k_{\cB }^j$.
Recall that $N-1/2>(n+1)/2$.  Then,
by Proposition~\ref{p:kernest}~\rmi\ there exists a constant $C$ such that 
$$
\sup_{y\in M}  \bignorm{k_{\cJ (t\cD)}(\cdot,y)}{2}
\leq C\,  \mod{t}^{-n/2}\, \bigl(1+\mod{t}\bigr)^{n/2-\rho} \qquad\forall t\in\BR^+. 
$$
Furthermore, by Lemma~\ref{l:splitf}, there exists a constant $C$ such that 
for $h$ in $\{0, \ldots, N\}$
$$
\bigmod{\cO^h(\om_j\, \wh m)(t)}
\leq C \, \norm{m}{\bS_{\be};J}\, 
\e^{-\be \mod{t}} \, \mod{t}^{h-J} \,\quant t \in \BR\setminus\set{0};
$$
here we have used the fact that $J\geq N+2$.
Since $\cO^h(\omega_j\widehat{m})$ vanishes in $[2-j,j-2]$,
$$
\begin{aligned}
\sup_{y\in M} \bignorm{k_{\cB } (\cdot,y)}{1}
&  \leq C\, \norm{m}{\bS_{\be};J}\,\sum_{\ell = 1}^\infty \ell^{\al/2} 
      \, \e^{\be \ell} \, 
      \sum_{j=\ell}^\infty \int_{\mod{t} \geq j-2} 
      \e^{-\be \mod{t}} \, \mod{t}^{N-J-\rho}  \wrt t \\
&  \leq C\, \norm{m}{\bS_{\be};J}\, \sum_{\ell = 1}^\infty \ell^{\al/2+N-J-\rho}.
\end{aligned}
$$
The series above is convergent, because $\al/2+N-J-\rho<-1$ by assumption, \break so
that 
$$
\sup_{y\in M} \bignorm{k_{\cB } (\cdot,y)}{1}
\leq C\, \norm{m}{\bS_{\be};J}.  
$$
Since $k_{\cB }(x,y) = \OV{k_{\cB }(y,x)}$, 
$$
\sup_{x\in M} \bignorm{k_{\cB } (x,\cdot)}{1}
\leq C\, \norm{m}{\bS_{\be};J}.  
$$
Hence $\cB $ maps $\lu{M}$ into $\lu{M}$, with operator norm
bounded by $C \, \norm{m}{\bS_{\be};J}$, as required to to prove (\ref{f: M2}).
\par 
The proof of the theorem is complete. 
\end{proof}

\end{document}